\documentclass[12pt]{amsart}
\usepackage{amssymb}

\newtheorem{theorem}{Theorem}[section]

\newtheorem{lemma}{Lemma}[section]
\newtheorem{cor}{Corollary}[section]

\title[Bounds for approximate discrete tomography]
{Bounds for approximate \\ discrete tomography solutions}

\author{Lajos Hajdu}
\address{Institute of Mathematics\\
University of Debrecen\\
H-4010 Debrecen, P.O. Box 12\\
Hungary}
\email{hajdul@science.unideb.hu}

\author{Rob Tijdeman}
\address{Mathematical Institute\\
Leiden University\\
2300 RA Leiden, P.O. Box 9512\\
The Netherlands}
\email{tijdeman@math.leidenuniv.nl}



\thanks{Research supported in part by the OTKA grants K75566, K100339 and NK101680, and by the T\'AMOP 4.2.1./B-09/1/KONV-2010-0007 project. The project is implemented through the New Hungary Development Plan, cofinanced by the European Social Fund and the European Regional Development Fund.}

\begin{document}
\begin{abstract}
In earlier papers we have developed an algebraic theory of discrete tomography. In those papers the structure of the functions $f: A \to \{0,1\}$ and
$f: A \to \mathbb{Z}$ having given line sums in certain directions have been analyzed. Here $A$ was a block in $\mathbb{Z}^n$ with sides parallel to the axes.
In the present paper we assume that there is noise in the measurements and (only)  that $A$ is an arbitrary or convex finite set in $\mathbb{Z}^n$.
We derive generalizations of earlier results. Furthermore we apply a method of Beck and Fiala to obtain results of he following type:
if the line sums in $k$ directions of a function $h: A \to [0,1]$ are known, then there exists a function $f: A \to \{0,1\}$ such that its line sums differ by at most $k$ from the corresponding line sums of $h$.
\end{abstract}

\maketitle

\section{Introduction}

Let $n$ be a positive integer and let $A$ be a finite subset of $\mathbb{Z}^n$. If $f: A \to \mathbb{R}$, then the line sum of $f$ along the line $l=\underline{c}+t\underline{d}$ (with $\underline{c},\underline{d}\in \mathbb{Z}^n$, $\underline{d}\neq \underline{0}$ fixed and $t\in \mathbb{R}$ variable) is defined as $\sum_{\underline{a} \in A \cap l} f(\underline{a})$. We call $\underline{d}$ a direction.
Let $S=\{\underline{d_1}, \dots, \underline{d_k}\}$ be a set of directions. By the line sums of $f$ along $S$ we mean all the line sums of $f$ along a line in a direction from $S$ passing through at least one point of $A$. Theorem 1 of \cite{ht4} states that if $A$ is a block with sides parallel to the axes, then any function $f: A \to \mathbb{R}$ with zero line sums along $S$ can be uniquely written as a linear combination
of so-called switching components of $S$ contained in $A$. In Section 3 we prove that for this result it suffices that $A$ is convex, but that the convexity requirement cannot be dropped.

By a discrete tomography problem we mean asking for a function $f:A \to Z$ which satisfies prescribed line sums along $S$, where $Z$ may be $\{0,1\}, \mathbb{R}, \mathbb{Z}$ or some finite real set.
The authors and others have developed an algebraic theory of the structure of the solutions of a discrete tomography problem, see \cite{ht1}, \cite{ht2}, \cite{ht3}, \cite{ht4}, \cite{h}, \cite{st}, \cite{sb}, \cite{dht}, \cite{bfht}.
It appears that the real solutions of a discrete
tomography problem form a linear manifold if there is at least one real solution, and that the integer solutions form a grid in this linear manifold, provided that at least one integer solution exists.

If the line sums are measured with some noise, then it is not certain that some function satisfies the measured line sums along $S$. A natural question is then what the best approximative solution is.
We shall show
that there is some linear manifold which can be considered as the set of `best real approximations' in the sense of least squares. An obvious choice is then to choose the shortest best approximation, that is the orthogonal projection of the origin to that linear manifold. In Section 4 we present an algorithm to construct this shortest best approximation and illustrate it by an example. In Section 5 we present an explicit system of linear equations which determines the shortest best solution in case $A$ is convex. As an application we generalize a result from \cite{dht} by giving an explicit expression for the shortest best solution in case $A$ is a rectangle with sides parallel to the axes and only row and column sums are given.

In the 80's Beck and Fiala \cite{bf} proved a `balancing' theorem. In Section 6 we show that this implies that if the line sums in $k$ directions of a function $h: A \to [0,1]$ are known, then there exists a function $f: A \to \{0,1\}$ such that its line sums differ by at most $k$ from the corresponding line sums of $h$.

We extend this result in Section 7 to the case that we are not searching for a binary image, but for an image $f$ with a finite number of given real values. To do so we generalize the result of Beck and Fiala.

\section{Notation}

We use the following notation throughout the paper. Let $n$ be a positive integer. For brevity, for $x_1, \dots, x_n \in \mathbb{R}$ and $u_1, \dots, u_n \in \mathbb{Z}$, we write $\underline{x} = (x_1, \dots,x_n)$, $\vec{x} = (x_1, \dots, x_n)^T$ and
$\underline{x}^{\underline{u}} = \prod_{j=1}^n x^{u_j}_j$.

Let $\underline{d} \in \mathbb{Z}^n$ with gcd$(d_1, \dots, d_n)=1$ be such that $\underline{d} \not= \underline{0}$, and for the smallest $j$ with $d_j \not = 0$ we have $d_j>0$. We call $\underline{d}$ a direction.
By lines with direction $\underline{d}$ we mean lines of the form $\underline{c}+t\underline{d}$ (with $\underline{c}\in \mathbb{Z}^n$ fixed, $t\in\mathbb{R}$ variable).

Let $A$ be a finite subset of $\mathbb{Z}^n$. Write $A = \{ \underline{a_1}, \dots, \underline{a_s}\}$ where $\underline{a_1}, \dots, \underline{a_s}$ are arranged in lexicographic increasing order.
We call $A$ convex if every $ \underline{a} \in \mathbb{Z}^n$ which belongs to the closed convex hull of $A$ belongs to $A$ itself. By the minimal corner of a set $B \subseteq \mathbb{Z}^n$ we mean the lexicographically smallest element $\phi(B)$ of $ B$.

If $f: A \to \mathbb{R}$, then the line sum of $f$ along the line
$l = \underline{c} + t \underline{d}$ is defined as $\sum_{\underline{a} \in A \cap l} f(\underline{a})$.
For any
$f:\ A\to{\mathbb R}$, write $\vec{f}:=
\left(f(\underline{a}_1),\dots,f(\underline{a}_s)\right)^T$.
We often identify $f$ and $\vec{f}$. The length of $\vec{f}$ (or $f$) is defined as $|f| = |\vec{f}| = \sqrt{\sum_{\underline{a} \in A} (f( \underline{a}))^2}$.

Let $k$ be a positive integer and $S=\{\underline{d_1}, \dots, \underline{d_k}\}$ be a fixed set of directions. By the line sums along $S$ we mean all the line sums along lines in a direction from $S$ which pass through at least one point of $A$. For $\underline{d}=(d_1,\dots,d_n)\in S$  put
$$
f_{\underline{d}}( \underline{x}) = (\underline{x}^{\underline{d}} -1) \prod_{d_j < 0} x_j^{-d_j},
$$
$F(\underline{x}) = \prod_{i=1}^k f_{\underline{d_i}} (\underline{x})$ and, for $ \underline{u} \in \mathbb{Z}^n$, set $F_{\underline{u}}(\underline{x}) = \underline{x}^{\underline{u}} F(\underline{x})$.
Obviously, the polynomial $F_{\underline{u}}$ has integer coefficients. We call the functions $F_{\underline{u}}$ the switching polynomials of $S$. Define the functions $m_{\underline{u}} : \mathbb{Z}^n \to \mathbb{Z}$ by
$$ m_{\underline{u}} ( \underline{v}) = {\rm coeff} (\underline{x}^{\underline{v}})~{\rm in} ~ F_{\underline{u}} (\underline{x}) ~ {\rm for} ~ \underline{v} \in \mathbb{Z}^n.$$
We define $D_{\underline{u}}$ as the set of $\underline{v} \in \mathbb{Z}^n$ for which $m_{\underline{u}} (\underline{v}) \not= 0$ and call it a switching component.
Let $\phi(\underline{u})$ denote the minimal corner of $D_{\underline{u}}$.
It follows from the above definitions that
\begin{equation} \label{phi}
 m_{\underline{u}}(\phi(\underline{u})) = \pm 1.
 \end{equation}

\section{The structure of functions with zero line sums}

We prove that Theorem 1 of \cite{ht4} remains true under the weaker condition that $A$ is convex.

\begin{theorem}
\label{base}
Let $A$ be a finite convex subset of  $\mathbb{Z}^n$, and $S$ a given set of directions. Then any function $f: A \to \mathbb{R}$ with zero line sums along $S$ can be uniquely written in the form
$$
f = \sum_{D_{\underline{u}} \subseteq A} c_{\underline{u}} m_{\underline{u}}
$$
with coefficients $c_{\underline{u}} \in \mathbb{R}$. Moreover, every such function $f$ has zero line sums along $S$.
\end{theorem}

\noindent If there is no $\underline{u}$ for which $D_{\underline{u}} \subseteq A$, then the only function $f$ with zero line sums along $S$ is the trivial function $f=0$.
Otherwise, the functions with zero line sums along $S$ form a proper linear subspace of the linear space of all functions $f: A \to \mathbb{R}$.

\begin{proof}
The statement has been proved in case $A$ is a hyperblock with sides parallel to the axes in Theorem 1 of \cite{ht4}.
Let $A^*$ be a hyperblock with sides parallel to the axes such that $A \subseteq A^*$. Set $f(\underline{x}) = 0$ for $\underline{x} \in A^* \setminus A$. Then we know that
\begin{equation}
\label{U}
f =  \sum_{D_{\underline{u}} \subseteq A^*} c_{\underline{u}} m_{\underline{u}}
\end{equation}
with coefficients $c_{\underline{u}} \in \mathbb{R}$.
It remains to prove that $c_{\underline{u}} = 0$ if $ D_{\underline{u}}$ is not contained in $A$.

If $ D_{\underline{u}}$ is not contained in $A$, then there exists $\psi(\underline{u})\in D_{\underline{u}}$ such that $\psi(\underline{u})\notin A$. Since $A$ is convex, there is a linear manifold $L$ which
extends a hyperface of the convex hull of $A$
such that $\psi(u)$ and $A$ are on different sides of $L$. Let $H_L$ be the open halfspace generated by $L$ which contains $\psi(u)$. Note that $H_L$ does not contain any element of $A$. Consider the set $U_L$ of all $\underline{u}$ such that $D_{\underline{u}} \subseteq A^*$ and $D_{\underline{u}}$ contains an element $\psi(\underline{u})\in H_L$. Without loss of generality we assume that $\psi(\underline{u})$ has maximal Euclidean distance $d(\psi(\underline{u}),L)$ to $L$
among the elements of $D_{\underline{u}}\cap H_L$ and, if there are more such elements with maximal distance to $L$, then $\psi(u)$ is the lexicographically smallest among them.
Since the sets $D_{\underline{u}}$ for variable $\underline{u}$ are translates of each other, the vectors $\psi(\underline{u}) -\underline{u}$ are the same for all $\underline{u} \in U_L$. Now we arrange the elements of $U_L$ according to the non-increasing distances $d(\psi(\underline{u}),L)$ of $\psi(\underline{u})$ to $L$. Thereafter we order the elements of $U_L$ for which the distances $d(\psi(\underline{u}),L)$ are equal according to non-decreasing lexicographic order of $\underline{u}$. Consider the first element $\underline{u} \in U_L$ according to this ordering. By the above construction there is no other set $D_{\underline{u}}$ for $\underline{u} \in U_L$ which contains $\psi(\underline{u})$.
Since $\psi(u) \notin A$ we infer $f(\psi(\underline{u})) = 0$, hence $c_{\underline{u}} =0$. We proceed with the next element $\underline{u} \in U_L$ in the ordering and conclude by a similar reasoning that $c_{\underline{u}} = 0$ also for this $\underline{u}$. Continuing until we have had all elements of $U_L$, we conclude that $c_{\underline{u}}= 0$ for all $\underline{u} \in U_L$. Since $D_{\underline{u}}$ was an arbitrary set not contained in $A$, the first statement follows. The uniqueness and the second statement of the theorem follow immediately from Theorem 1 of \cite{ht4}.
\end{proof}

The following result is a consequence of Theorem \ref{base}.

\begin{cor}
In the notation of Theorem \ref{base}, for any $h: A\to \mathbb{R}$ and for any prescribed values from $\mathbb{R}$ at the minimal corners of the switching components contained in $A$ there exists a unique $f: A\to\mathbb{R}$ having the same line sums along $S$ as $h$ has and having the prescribed values at the minimal corners. Moreover, if $h: A \to \mathbb{Z}$, then $f: A \to \mathbb{Z}$.
\end{cor}

\begin{proof}
According to Theorem \ref{base} there are unique coefficients $c_{\underline{u}} $ such that
$$
f = h + \sum_{D_{\underline{u}} \subseteq A} c_{\underline{u}} m_{\underline{u}}
$$
has the same line sums along $S$ as $h$. By (\ref{phi}) we obtain, following the ordering argument from the previous proof, that each coefficient $c_{\underline{u}} $ is completely determined by the value of $m_{\underline{u}}$ at $\phi({u})$ and, moreover, that $c_{\underline{u}} \in \mathbb{Z}$ if $h: A \to \mathbb{Z}$.
\end{proof}

\noindent {\bf Remark 3.1.} The following example shows that in Theorem 3.1 we cannot drop the convexity requirement. \\
Let $A = \{ (0,0), (0,1), (1,0), (1,2), (2,1), (2,2) \} $ and $S = \{ (1,0), (0,1) \}$.
Then for every $ \underline {u}$ we have $D_{\underline{u}} - \underline{u} = \{ (0,0), (0,1), (1,0), (1,1) \}.$ Therefore $A$ does not contain any switchting component.
However, there is a nontrivial function $f: A \to \mathbb{Z}$ with all line sums along $S$ equal to 0: \\
$f(0,0)=1, f(0,1) =-1, f(1,0) = -1, f(1,2) = 1, f(2,1) = 1, \\f(2,2) = -1.$

\section{The best approximating function for general domains}

The next theorem can be used to construct the function $f_0: A \to \mathbb{R}$ such that $f_0$ fits optimally the measured line sums along $S$ in the sense of least squares and, moreover, has minimal Euclidean length among such functions.

Let $A\subseteq{\mathbb Z}^n$ be a finite, nonempty set, and write $\underline{a}_1,\dots,\underline{a}_s$ for its elements. For the rest of this section, fix the indexing of the elements.

Let $B$ be a $t$ by $s$ matrix of real numbers.
The range of the matrix $B$ is denoted by
$$
R(B):=\{B\cdot\vec{x}\ :\ \vec{x}\in{\mathbb R}^s\}.
$$
Hence $R(B)$ is a subspace of ${\mathbb R}^t$, generated by the column vectors $\vec{b}_1,\dots,\vec{b}_s$ of $B$. We have $0\leq {\text{dim}}(R(B))\leq t$.
Write $B_1$ for a matrix formed by a maximal linearly independent set of column vectors of $B$. Then $B_1 = B \cdot C_1$ where $C_1$ is a matrix of type $s\times ({\rm{rank}}(B))$ which has rank$(B)$ entries 1 in distinct columns and all other entries equal to 0. Observe that $B_1^T \cdot B_1$ is invertible.

\begin{lemma}
\label{lemlstsq}
Let $A, B$ and $B_1$ be as above. Let $\vec{b}\in{\mathbb R}^t$ be arbitrary. Put
\begin{equation}
\label{b^*}
\vec{b^*} = B_1 \cdot (B_1^T \cdot B_1)^{-1} \cdot B_1^T \cdot \vec{b}.
\end{equation}
Then $\vec{b^*}$ is the vector from $R(B)$ which is closest to $\vec{b}$.
\end{lemma}

\begin{proof}
Obviously, the vector $\vec{b}^*$ in $R(B)$ closest to $\vec{b}$ is uniquely determined by the following properties:
\begin{itemize}
\item $\vec{b}^*\in R(B)$,
\item $\vec{b}-\vec{b}^*$ is orthogonal to $R(B)$.
\end{itemize}
Since $B_1 = B \cdot C_1$ the first property follows immediately from \eqref{b^*}.
The second property is equivalent to that $\vec{b}-\vec{b}^*$ is orthogonal to all the column vectors of $B$, or equivalently, to all column vectors of $B_1$. In other words, it is equivalent to
$$
B_1^T\cdot(\vec{b}-\vec{b}^*)=\vec{0}.
$$
It follows from \eqref{b^*} that
$$
B_1^T\cdot \vec{b}^* = B_1^T\cdot B_1 \cdot (B_1^T \cdot B_1)^{-1} \cdot B_1^T \cdot \vec{b} = B_1^T \cdot \vec{b}.
$$
Hence both properties are satisfied.
\end{proof}

We use the above notation and define
$$
\vec{l}_f = B\cdot \vec{f}.
$$
Let  $B_2$ be a matrix formed by a maximal linearly independent set of row vectors of $B$. Then $B_2 = C_2 \cdot B$ where $C_2$ is a matrix of type $({\rm{rank}}(B))\times t$ which has rank$(B)$ entries 1 in distinct rows and all other entries equal to 0. Observe that $B_2 \cdot B_2^T$ is invertible.

\begin{theorem}
\label{thmlstsq}
Let $A, B, B_2, C_2, \vec{b}, \vec{b^*}, f$ ($ = \vec{f}$) be as above. Put
\begin{equation} \label{f_0}
\vec{f_0} = B_2^T \cdot (B_2 \cdot B_2^T)^{-1} \cdot C_2 \cdot \vec{b^*}
\end{equation}
Then the corresponding $f_0: A \to \mathbb{R}$ has the following properties:
\begin{itemize}
\item[(i)] for any $f:\ A\to{\mathbb R}$ we have $|\vec{l}_f-\vec{b}|\geq |\vec{l}_{f_0}-\vec{b}|$,
\item[(ii)] if $f:\ A\to{\mathbb R}$, $f\neq f_0$ such that $|\vec{l}_f-\vec{b}|=|\vec{l}_{f_0}-\vec{b}|$, then $|\vec{f}|>|\vec{f}_0|$.
\end{itemize}
\end{theorem}

\begin{proof}
Observe that (i) and (ii) are equivalent with the following two properties:
\begin{itemize}
\item $B\cdot \vec{f}_0=\vec{b}^*$,
\item $\vec{f_0}$ is orthogonal to ker$(B)$, the nullspace of $B$.
\end{itemize}
The first property is clearly equivalent to
\begin{equation}
\label{eqforf01}
B_2 \cdot \vec{f}_0= C_2 \cdot \vec{b^*}.
\end{equation}
It follows immediately from \eqref{f_0} that this property is satisfied.
Since
$$
{\rm{ker}}(B)=\{\vec{x}\in{\mathbb R}^s\ :\ B\cdot\vec{x}=\vec{0}\},
$$
the orthogonal complement of ker$(B)$ is the subspace of ${\mathbb R}^s$ generated by the row vectors of $B$. Hence, by the definition of $B_2$, the second property above is equivalent to
\begin{equation}
\label{eqforf02}
\vec{f_0}=B_2^T\cdot\vec{y}\ \ \ {\rm{for\ some}}\ \vec{y}\in{\mathbb R}^r,
\end{equation}
where $r$ is the rank of $B_2$.
This is obviously true because of \eqref{f_0}.
Thus both properties are satisfied.
\end{proof}

\noindent {\bf Remark 4.1.} An alternative version of Theorem \ref{thmlstsq} can be obtained by using the Moore-Penrose pseudo inverse, cf. the proof of Theorem 1 in \cite{bfht}.
\vskip.3cm

\noindent {\bf Remark 4.2.} We apply Lemma \ref{lemlstsq} and Theorem \ref{thmlstsq} in the context of Discrete Tomography as follows. Let $A$ be a finite subset of $\mathbb{Z}^n$ and $S$ a set of directions. Let $l_1, \dots, l_t$ be the measured line sums along $S$. Note that because of noise they need not be consistent. Then $B$ is the $s$ by $t$ matrix whose entry $B_{ij}$ equals $1$ if the line corresponding to $l_j$ passes through $\underline{a}_i$ and $0$ otherwise. The vector $\vec{b}^*$ constructed in Lemma \ref{lemlstsq} represents the corresponding line sums along $S$ which are consistent and provide the optimal choice in the sense that $\sum_{j=1}^t (l_j - b_j^*)^2$ is minimal among the consistent line sums $b^*_j$ along $S$. Furthermore, the vector $\vec{f_0}$ constructed in Theorem \ref{thmlstsq} is the shortest best approximation in the sense that it is the shortest vector realizing the line sums given by $\vec{b^*}$. The corresponding function $f_0: A \to \mathbb{R}$ may be considered as the optimal choice for the measured line sums $l_1, \dots, l_t$.
\vskip.1cm

We illustrate the method by an example.
\vskip.3cm

\noindent
{\bf Example.} We use the notation of Lemma \ref{lemlstsq} and Theorem \ref{thmlstsq}. Consider the following subset of ${\mathbb R}^2$:
$$
A:=\{(1,0), (3,0), (0,1), (4,1), (0,2), (4,2), (1,3), (2,3), (3,3)\}.
$$
As the set of directions, take
$$
S:=\{(1,0), (0,1), (1,-1), (1,1)\}.
$$
The ordering of the points in $A$ and directions in $S$ are arbitrary, but fixed.
As a (measured) line sum vector, take
$$
\vec{b}^T:=\left(1, \frac{23}{10}, \frac75, 1, 1, 1, \frac32, 1, \frac65, 1, 1, 1, \frac9{10}, \frac{13}{10}, \frac12, 1, \frac65, \frac35, \frac12, \frac{17}{10}, \frac7{10}\right).
$$
The entries of $\vec{b}$ belong to the lines
$$  y=t\ \ (t=0,1,2,3),\ \ \ x=t\ \ (t=0,1,2,3,4) $$
$$ y=x+t\ \ (t=-3,-2,-1,0,1,2),\ \ \ y=-x+t\ \ (t=1,2,3,4,5,6) $$
which we keep in this order.
Then the matrix $B$ of line sums is given by
$$
{\footnotesize
\left(
{\begin{array}{rrrrrrrrr}
1 & 1 & 0 & 0 & 0 & 0 & 0 & 0 & 0 \\
0 & 0 & 1 & 1 & 0 & 0 & 0 & 0 & 0 \\
0 & 0 & 0 & 0 & 1 & 1 & 0 & 0 & 0 \\
0 & 0 & 0 & 0 & 0 & 0 & 1 & 1 & 1 \\
0 & 0 & 1 & 0 & 1 & 0 & 0 & 0 & 0 \\
1 & 0 & 0 & 0 & 0 & 0 & 1 & 0 & 0 \\
0 & 0 & 0 & 0 & 0 & 0 & 0 & 1 & 0 \\
0 & 1 & 0 & 0 & 0 & 0 & 0 & 0 & 1 \\
0 & 0 & 0 & 1 & 0 & 1 & 0 & 0 & 0 \\
0 & 1 & 0 & 1 & 0 & 0 & 0 & 0 & 0 \\
0 & 0 & 0 & 0 & 0 & 1 & 0 & 0 & 0 \\
1 & 0 & 0 & 0 & 0 & 0 & 0 & 0 & 0 \\
0 & 0 & 0 & 0 & 0 & 0 & 0 & 0 & 1 \\
0 & 0 & 1 & 0 & 0 & 0 & 0 & 1 & 0 \\
0 & 0 & 0 & 0 & 1 & 0 & 1 & 0 & 0 \\
1 & 0 & 1 & 0 & 0 & 0 & 0 & 0 & 0 \\
0 & 0 & 0 & 0 & 1 & 0 & 0 & 0 & 0 \\
0 & 1 & 0 & 0 & 0 & 0 & 0 & 0 & 0 \\
0 & 0 & 0 & 0 & 0 & 0 & 1 & 0 & 0 \\
0 & 0 & 0 & 1 & 0 & 0 & 0 & 1 & 0 \\
0 & 0 & 0 & 0 & 0 & 1 & 0 & 0 & 1
\end{array}}
 \right)
 }
 .
$$
As one can easily check, rank$(B)=9$. So we can take the matrix $C_1$ as the $9\times 9$ unit matrix. Thus $B_1=B$. Then, by \eqref{b^*}, the vector $\vec{b^*}^T$  is given by
$$
_{\left(\frac{891}{800}, \frac{2457}{1600}, \frac{1019}{800}, \frac{4361}{3200}, \frac{167}{128}, \frac{103}{128}, \frac{111}{128}, \frac{103}{128}, \frac{963}{640}, \frac{4239}{3200}, \frac{859}{1600}, \frac{1211}{1600}, \frac{1433}{3200}, \frac{287}{200}, \frac{2511}{3200}, \frac{4239}{3200}, \frac{1179}{1600}, \frac{571}{1600}, \frac{153}{3200}, \frac{367}{200}, \frac{3151}{3200}\right).}
$$
As one can readily check, the indices of a maximal set of independent rows of $B$ is given by
$$
\{1,2,3,4,5,6,7,10,11\}.
$$
That is, we may take
$$
{\footnotesize
C_2:=
 \left(
{\begin{array}{rrrrrrrrrrrrrrrrrrrrr}
1 & 0 & 0 & 0 & 0 & 0 & 0 & 0 & 0 & 0 & 0 & 0 & 0 & 0 & 0 & 0 & 0 & 0 & 0 & 0 & 0\\
0 & 1 & 0 & 0 & 0 & 0 & 0 & 0 & 0 & 0 & 0 & 0 & 0 & 0 & 0 & 0 & 0 & 0 & 0 & 0 & 0\\
0 & 0 & 1 & 0 & 0 & 0 & 0 & 0 & 0 & 0 & 0 & 0 & 0 & 0 & 0 & 0 & 0 & 0 & 0 & 0 & 0\\
0 & 0 & 0 & 1 & 0 & 0 & 0 & 0 & 0 & 0 & 0 & 0 & 0 & 0 & 0 & 0 & 0 & 0 & 0 & 0 & 0\\
0 & 0 & 0 & 0 & 1 & 0 & 0 & 0 & 0 & 0 & 0 & 0 & 0 & 0 & 0 & 0 & 0 & 0 & 0 & 0 & 0\\
0 & 0 & 0 & 0 & 0 & 1 & 0 & 0 & 0 & 0 & 0 & 0 & 0 & 0 & 0 & 0 & 0 & 0 & 0 & 0 & 0\\
0 & 0 & 0 & 0 & 0 & 0 & 1 & 0 & 0 & 0 & 0 & 0 & 0 & 0 & 0 & 0 & 0 & 0 & 0 & 0 & 0\\
0 & 0 & 0 & 0 & 0 & 0 & 0 & 0 & 0 & 1 & 0 & 0 & 0 & 0 & 0 & 0 & 0 & 0 & 0 & 0 & 0\\
0 & 0 & 0 & 0 & 0 & 0 & 0 & 0 & 0 & 0 & 1 & 0 & 0 & 0 & 0 & 0 & 0 & 0 & 0 & 0 & 0
\end{array}}
\right)
}
$$
whence
$$
{\footnotesize
B_2=
 \left(
{\begin{array}{rrrrrrrrr}
1 & 1 & 0 & 0 & 0 & 0 & 0 & 0 & 0 \\
0 & 0 & 1 & 1 & 0 & 0 & 0 & 0 & 0 \\
0 & 0 & 0 & 0 & 1 & 1 & 0 & 0 & 0 \\
0 & 0 & 0 & 0 & 0 & 0 & 1 & 1 & 1 \\
0 & 0 & 1 & 0 & 1 & 0 & 0 & 0 & 0 \\
1 & 0 & 0 & 0 & 0 & 0 & 1 & 0 & 0 \\
0 & 0 & 0 & 0 & 0 & 0 & 0 & 1 & 0 \\
0 & 1 & 0 & 1 & 0 & 0 & 0 & 0 & 0 \\
0 & 0 & 0 & 0 & 0 & 1 & 0 & 0 & 0
\end{array}}
 \right)
.}
$$
Finally, by \eqref{f_0} we obtain
$$
{\footnotesize
\vec{f_0}^T=
\left(\frac{1211}{1600}, \frac{571}{1600}, \frac{1817}{3200},
\frac{3097}{3200}, \frac{1179}{1600}, \frac{859}{1600}, \frac{153}{3200}, \frac{111}{128}, \frac {1433}{3200}
\right)
.}
$$

\section{The best approximating function for convex domains}

The following theorem provides explicitly a system of linear equations which determines the best approximating function constructed in the previous section. We illustrate in the corollary the advantage of this explicit expression. The real number $l(Y_{\tau})$ in the following theorem can be considered as the measured line sum of $f$ along the line corresponding to $Y_{\tau}$.

\begin{theorem} \label{central}
Let $A\subseteq \mathbb{R}^n$ be convex. Let $S$ be a finite set of directions and $Y_1, \dots, Y_t$ the subsets of $A$ which determine the lines along $S$. Suppose for $\tau=1, \dots, t$ a real number $l(Y_{\tau})$ is given. Let $U_A \subset A$ be the set of minimal corners of the switching components contained in $A$. Define $f_0: A \to \mathbb{R}$ by the system of linear equations
\begin{equation}
\label{orth}
\sum_{\underline{v} \in D_{\underline{u}}} f_0(\underline{v}) m_{\underline{u}}(\underline{v})  = 0 ~~{\it for~all}~ \underline{u}~{\it with}~ \phi(\underline{u}) \in U_A,
\end{equation}
\begin{equation}
\label{linear}
\sum_{\tau: \underline{u} \in Y_{\tau}} \sum_{\underline{v} \in Y_{\tau}} f_0(\underline{v}) = \sum_{\tau: \underline{u} \in Y_{\tau}} l(Y_{\tau}) ~{\it for~all}~ \underline{u}~{\it with}~ \phi(\underline{u}) \in A \setminus  U_A.
\end{equation}
Then $f_0$ is a function such that
\begin{equation}
\label{linman}
\sum_{\tau=1}^t  \left( \sum_{\underline{v} \in Y_{\tau}} f_0(\underline{v}) - l(Y_{\tau})  \right)^2
\end{equation}
is minimal and among such functions $f_0$ is the one for which the value of $|\vec{f_0}|$ is minimal.
\end{theorem}

\begin{proof}
By Theorem \ref{thmlstsq} the function $f_0: A \to \mathbb{R}$ satisfying (\ref{linman}) for which $|\vec{f_0}|^2 = \sum_{\underline{v} \in A} (f_0(\underline{v}))^2$ is minimal is uniquely determined. We proceed with this function $f_0$ and consider it as  a function for which each value $f_0(\underline{u})$ for $ \underline{u} \in A$ is a variable.
It follows by differentation of \eqref{linman} to each $f_0(\underline{u})$ that
$$
\sum_{\tau: \underline{u} \in Y_{\tau}} \sum_{\underline{v} \in Y_{\tau}} f_0(\underline{v}) = \sum_{\tau: \underline{u} \in Y_{\tau}}l(Y_{\tau})
$$
for all $\underline{u} \in A$. Hence $f_0$ satisfies \eqref{linear}.

We know that $\vec{f_0}$ is orthogonal to the linear subspace $L$ of functions having zero line sums along $S$. According to Theorem \ref{base} the functions $m_{\underline{u}}$ have zero line sums along $S$. Therefore they are in $L$ for all $\underline{u} \in \mathbb{Z}^n$. Since the inner product of $\vec{f_0}$ and any vector from $L$ is $0$, $f_0$ satisfies \eqref{orth} too.

The numbers of linear equations in \eqref{orth} and
\eqref{linear} together equal the cardinality of $A$. Thus it suffices to show that they are linearly independent over $\mathbb{R}$ in order to prove that $f_0$ is completely determined by them. Because of the orthogonality of $\vec{f_0}$ and $L$, it is enough to prove that the equations in \eqref{orth} are linearly independent as well as those in \eqref{linear}.

Since by Theorem \ref{base} the functions $m_{\underline{u}}$ are linearly independent, the equations \eqref{orth} are linearly independent as well.

Furthermore, in Theorem \ref{base} it is shown that $f_0$ is uniquely determined by its values at $U_A$. This shows that the equations in \eqref{linear} are linearly independent. We conclude that the linear equations in \eqref{orth} and \eqref{linear} are linearly independent indeed.
\end{proof}

In the particular case that $A \subset \mathbb{Z}^2$ is a rectangular block, and we only have row and column sums, we give an explicit form of $f_0$.
The result shows that the formula from \cite{dht} is also valid if there is noise in the measurements. We simplify our notation.

\begin{cor}
\label{dht}
Let $A =\{(i,j)\in{\mathbb Z}^2:0\leq i<q,0\leq j<p\}, S=\{(1,0),(0,1)\}$.
Let $c_i$ $(i=0,\dots,q-1)$ and $r_j$ $(j=0,\dots,p-1)$ denote the measured column sums and row sums, respectively. Further, write $s_r=\sum\limits_{j=0}^{p-1} r_j, s_c=\sum\limits_{i=0}^{q-1} c_i$ and $T =
\frac {ps_r + qs_c}{q+p}$.
Then for any $(i,j)\in A$ we have
$$
f_0(i,j)= \frac{c_i}{p} + \frac{r_j}{q} - \frac{T}{qp}.
$$
\end{cor}

\noindent Observe that if $s_r = s_c$, then $T = s_r = s_c $.

\begin{proof}
Since
$$(\frac{r_j}{q} + \frac{c_i}{p} - \frac {T}{qp}) - (\frac{r_j}{q} + \frac{c_{i+1}}{p} - \frac {T}{qp}) - (\frac{r_{j+1}}{q} + \frac{c_i}{p} - \frac {T}{qp}) + (\frac{r_{j+1}}{q} + \frac{c_{i+1}}{p} - \frac {T}{qp}) = 0$$
for all $i$ and $j$, the equations \eqref{orth} are satisfied.
Furthermore
$$ (\frac{r_1}{q} + \frac{c_i}{p} - \frac {T}{qp}) + \dots + (\frac{r_p}{q} + \frac{c_i}{p} - \frac {T}{qp}) + (\frac{r_j}{q} + \frac{c_1}{p} - \frac {T}{qp}) + \dots + (\frac{r_j}{q} + \frac{c_n}{p} - \frac {T}{qp})$$
$$ = \frac{s_r}{q} + c_i - \frac{T}{p} +r_j + \frac{s_c}{p} - \frac{T}{q} = c_i + r_j,$$
which shows that the equations \eqref{linear} are also satisfied.
\end{proof}

\section{Approximate solutions in the binary case}

Let $A$ be a finite subset of $ \mathbb{Z}^n$.
We assume that a function $h: A \to \mathbb{R}$ is given and provide information on the `nearest' function $f: A \to \mathbb{Z}$ having approximately the same line sums along $S$ as $h$.

If $n=2$ and only row and column sums are given, we have the following result.

\begin{theorem} \label{bara}
If $h: A \to \mathbb{R}$ is given, there exists a function $f: A \to \mathbb{Z}$ such that every two corresponding elements of $f$ and $h$ as well as every two corresponding row sums and column sums as well as the sums of all function values of $f$ and $h$ differ by less than $1$.
\end{theorem}

We apply the following result of Baranyai.

\begin{lemma}[\cite{bar}, Lemma 3]\label{bar}
Let $[h_{ij}]$ be an $l$ by $m$ matrix of real elements. Then there exists an $l$ by $m$ integer matrix $[f_{ij}]$ such that
$$|h_{ij} - f_{ij}| <1 {\rm ~~for~all~}i,j,$$
$$ | \sum_i h_{ij} - \sum_i f_{ij} |  <1 {\rm~~for~all~}j,$$
$$ | \sum_j h_{ij} - \sum_j f_{ij} | < 1 {\rm~~for~all~}i,$$
$$ | \sum_i \sum_j h_{ij} - \sum_i \sum_j f_{ij} | <1.$$
\end{lemma}

\begin{proof}[Proof of Theorem \ref{bara}]
Choose an $l$ by $m$ block $A^*$ which covers $A$. For $(i,j)\in A^*\setminus A$ put $h(i,j)=0$. This does not change the line sums. Applying Lemma \ref{bar}, we get $f(i,j)=h(i,j)=0$ for $(i,j) \in A^*\setminus A$ and the theorem follows.
\end{proof}

\noindent The following example shows that the bound 1 is best possible. Let $ 0<\varepsilon<1$, $l> 1/ \varepsilon$, $m=1$, $h(i,1) = \varepsilon$ for $i=1, \dots, l$. Then $f(i,1) =1$ for some $i$ in order to avoid that the row sums of $h$ and $f$ differ more than 1. But then the $i$-th column sums of $h$ and $f$ have a difference $1 - \varepsilon$.
\vskip.1cm

The crucial feature of the following general result is that the upper bound is independent of the size of $A$.

\begin{theorem}
\label{BF}
Let $A$ be a finite set in $\mathbb{Z}^n$. Let $h: A \to \mathbb{R}$ and let $k$ directions $S$ be given. Then
there exists a function $f: A \to \mathbb{Z}$ such that each difference between corresponding elements of $h$ and $f$ is less than $1$
and each difference between corresponding line sums of $h$ and $f$ along $S$ is at most $k-1$.
\end{theorem}

We introduce the following notation in order to apply a result of Beck and Fiala. Let $X= \{x_1,x_2, \dots\}$ be a finite set and $\mathcal{F}$ a family of subsets of $X$. Associate to every $x_i$ a real number $\alpha_i$. Let $k$ be the degree of $\mathcal{F}$, that is the maximal number of elements of $\mathcal{F}$ to which some element of $X$ belongs. Let $r(k)$ be the least value for which one can find integers $a_i, ~i=1,2, \dots$ so that $|a_i - \alpha_i| < 1$ and
$$ | \sum_{x_i \in E} a_i - \sum_{x_i \in E} \alpha_i| \leq r(k) $$ for all $E \in \mathcal{F} $.
The following result is due to Beck and Fiala (see \cite{bf}). We shall prove a generalization of it in the next section.

\begin{lemma}
\label{bf}
In the above notation, we have
$$ r(k) \leq k-1~~{\rm for}~~k \geq 2.$$
\end{lemma}

\noindent Beck and Fiala conjecture that $r(k) \leq  k/2$ is true even for small values of $k$. Bednarchak and Helm \cite{bh} and Helm \cite{he}  improved the Beck-Fiala bound to $r(k) \leq k-3/2$ for $k \geq 3$ and $r(k) \leq k-2$ for $k$ sufficiently large, respectively.

\begin{proof}[Proof of Theorem \ref{BF}]
Let $Y_1, \dots, Y_t$ denote the subsets of $A$ which determine the line sums along $S$. Let $\mathcal{F}=\{Y_1,Y_2, \dots, Y_t\} $.
By Lemma \ref{bf} there exist integers $f(a)$ for all $ a \in A$ with $f(a) \in \{\lfloor h(a) \rfloor, \lceil h(a) \rceil \}$ such that $\sum_{a \in Y_j}  | f(a) - h(a) |  \leq k-1$
\noindent for $j= 1, \dots, t$.
\end{proof}

\noindent{\bf Remark 6.1}. Obviously, many variations of Theorem \ref{bf} are possible. E.g. adding the requirement that the sum of all values $f(a)$ differs little from the sum of all values $h(a)$ leads to an upper bound $k$. The requirement that the difference between the sums of the values of $f$ and $h$ along any linear manifold parallel to the axes should be small leads to an upper bound $2^k-2$.
\vskip.1cm

 \noindent {\bf Remark 6.2.} By a probabilistic method a better dependence on $k$ can be obtained at the cost of some dependence on $A$.
 An recent improvement by Banaszczyk \cite{ba} of a result of Beck implies that in Theorem \ref{BF} the upper bound $k-1$ can be replaced by $C\sqrt{k \log (\min (m,n))}$,
where $C$ is some constant.

\section{Approximate solutions for grey values}

\begin{theorem}
\label{GV}
Let $Z = \{z_1, \dots, z_m \}$ be a set of $m$ real numbers with $z_1<  \dots < z_m$. Put $z = \max_i (z_{i+1} - z_i)$.
Let $h: A \to [z_1,z_m]$ and let $k$ directions $S$ be given.
Then there exists a function $f: A \to Z$ such that the difference between the values of $f$ and $h$ at any element of $A$ is at most $z$
and each difference between corresponding line sums of $f$ and $h$ along $S$ is at most $(k-1)z$.
\end{theorem}

For the proof we derive the following extension of the lemma of Beck and Fiala.

\begin{lemma}
\label{bfe}
Let $Z = \{z_1, \dots, z_m \}$ be a set of $m$ real numbers with $z_1<  \dots< z_m$. Put $z = \max_i (z_{i+1} - z_i)$.
Let $X= \{x_1,x_2, \dots , x_s\}$ be a finite set and associate to every $x_i$ a real number $\alpha_i \in [z_1,z_m]$.
Then given any family $\mathcal{F}$ of subsets of $X$ having maximum degree $k \geq 2$, there exist $a_i \in Z$ such that
$a_i=z_j$ if $\alpha_i = z_j$ and there is no element from $Z$ in between $\alpha_i$ and $a_i$ for all $i$ and $j$
and
$$ \left| \sum_{x_i \in E} a_i - \sum_{x_i \in E} \alpha_i \right| \leq (k-1)z $$
for all $E \in \mathcal{F} $.
\end{lemma}

\begin{proof}

We shall define a sequence $\alpha^0, \alpha^1, \dots, \alpha^p$ of $s$-dimensional vectors $\alpha^j = (\alpha_1^j, \dots, \alpha_s^j)$ and a sequence $Y_j$ of subsets of $X$ with the following properties: \\
(i) $ \alpha_i^0 = \alpha_i$ for $i = 1, \dots, s.$ \\
(ii) There is no element of $Z$ in between $\alpha_i$ and $\alpha_i^j$ for $ i = 1, \dots, s; j = 0,1, \dots, p.$ \\
(iii) $X \setminus Y_j$ is a set of points $x$ for which $x \in Z$ for all $j$. \\
(iv) $Y_0 \supset Y_1 \supset \dots \supset Y_p $ and $|Y_j| = p-j$ for $0 \leq j \leq p$. \\
(v) $\alpha_i^j = \alpha_i^h$ for $j=h, \dots, p$ whenever $\alpha_i^h \in Z$. \\
(vi) If $|E \cap Y_j| > k$, then $\sum_{x_i \in E} \alpha_i^j = \sum_{x_i \in E} \alpha_i^{j+1}$ for all $E \in \mathcal{F}$.\\
(vii) For $j=0,1, \dots, p$ and all $E \in \mathcal{F}$ we have
$$ \left|\sum_{x_i \in E} \alpha_i^{j} - \sum_{x_i \in E} \alpha_i \right| \leq (k-1)z.$$
\noindent According to (iii) and (iv) the final vector $\alpha^p$ has all coordinates in $Z$.

We construct the sequence $(\alpha^j)$ by induction.
Suppose $\alpha^j$ is defined satisfying the above conditions for $j$. Let $$G_j = \{E \in \mathcal{F} : |E \cap Y_j| \geq k \}.$$
We distinguish between three cases. At every step there is some $i$ such that $x_i \in Y_j, \alpha_i^{j+1} \in Z$ and we set $Y_{j+1} = Y_j \setminus \{x_i\} $. \\
Case (a) $G_j = \emptyset.$ \\
Case (b) $ 0 < |G_j| < |Y_j|$. \\
Case (c) $|G_j| \geq  |Y_j|.$

Case (a). If $G_j$ is empty, then choose $\alpha_i^{j+1}$ as the element from $Z$ which is nearest to $\alpha_i$ for all $i$ with $x_i \in Y_j$. It follows that
$$
\left|\sum_{x_i \in E} \alpha_i - \sum_{x_i \in E} \alpha_i^{j+1}\right| \leq (k-1)z ~~{\rm for~ all}~~ E \in \mathcal{F},
$$
and the above conditions are satisfied for $j+1$. \\(It follows that $\alpha_i^{j+1} = \dots = \alpha_i^p = a_i$ for all $i$.)

In Case (b) associate a real variable $\beta_i$ to every $i = 1, \dots , s$ and consider the system of equations
$$ \sum_{x_i \in E \cap Y_j} \beta_i = 0 ~~{\rm for}~~ E \in G_j,$$
$$\beta_i = 0 ~~{\rm for}~~x_i \notin Y_j.$$
A nontrivial solution $\{ \beta_i \}_{i=1}^s$ exists, because in case (b) there are more variables than equations.
Let $t_0$ be the smallest nonnegative value for which $\alpha_i^j +t \beta_i \in Z$ for some $i$ with $x_i \in Y_j.$
Put $\alpha_i^{j+1} = \alpha_i^j + t_0 \beta_i$ for $i=1, \dots, s$.
It is easy to check that $$\sum_{x_i \in E} \alpha_i^{j} = \sum_{x_i \in E} \alpha_i^{j+1} ~~ {\rm for ~all}~~ E \in G_j.$$
\noindent Hence the above conditions are satisfied for $j+1$.

Case (c). Since each $x_i$ has degree at most $k$ in $G_j$, we may conclude that $|G_j| = |Y_j|$, each $x_i$ has degree exactly $k$ in $G_j$ and
$|E \cap Y_j| = k$ for every $E \in G_j$.
Let $\alpha_i^{j+1}$ be the element from $Z$ nearest to $\alpha_i$ for  every $x_i \in Y_j$.
Then $| \alpha_i^{j+1} - \alpha_i| \leq z/2$ for $x_i \in Y_j$. Since $k/2 \leq k-1$, we obtain
$$ |\sum_{x_i \in E} \alpha_i^{j+1} - \sum_{x_i \in E} \alpha_i| \leq (k-1)z$$
for all $E \in \mathcal{F}$. Hence the above conditions are satisfied for $j+1$. \\(It follows that $\alpha_i^{j+1} = \dots = \alpha_i^p = a_i$ for all $i$.)

Write $a_j =\alpha_i^p$ for $i = 1, \dots, s$. It is easy to check that in each case the relations (iii), (iv) and (vii) hold. This completes the proof.
\end{proof}

\begin{proof} [Proof of Theorem \ref{GV}]
Let $Y_1, \dots, Y_t$ denote the subsets of $A$ which determine the line sums along $S$.
By Lemma \ref{bfe} there exists a function $f: A \to Z$  such that
$$\sum_{a \in A \cap Y_j}  | f(a) - h(a) |  \leq (k-1)z$$ for 
$j= 1, \dots, t$.
\end{proof}

\noindent {\bf Remark 7.1.} A small adjustment must be made if the entries are not all in $[z_1,z_m].$ E.g. values of $h$ smaller than $z_1$ are first replaced by $z_1$, values larger than $z_m$ by $z_m$.
\vskip.1cm

\noindent {\bf Remark 7.2.} If we want to have relatively short vectors $\underline{f},\underline{g}$, then we may apply Theorem \ref{bfe} to the function $f_0$ from Theorem \ref{thmlstsq}.

\end{document}